\documentclass[11pt]{amsart} 
\usepackage{amssymb}
\usepackage{amsmath}
\usepackage{setspace}
\usepackage{amsthm}
\usepackage{epsfig,color,colordvi,wasysym}

\newtheorem{theorem}{Theorem}

\newtheorem{lemma}[theorem]{Lemma}

\newtheorem{proposition}[theorem]{Proposition}

\renewcommand\emptyset{\varnothing}

\begin{document}

\title{On Weak Chromatic Polynomials of Mixed Graphs}

\author{Matthias Beck}
\address{Department of Mathematics, San Francisco State University, San Francisco, CA 94132, USA}
\email{mattbeck@sfsu.edu}

\author{Daniel Blado}
\address{Department of Computing and Mathematical Sciences, California Institute of Technology, Pasadena CA 91125, USA}
\email{blado@caltech.edu}

\author{Joseph Crawford}
\address{Department of Mathematics, Morehouse College, Atlanta GA 30314, USA}
\email{josephdcrawford@yahoo.com}

\author{Ta\"ina Jean-Louis}
\address{Department of Mathematics, Amherst College, Amherst MA 01002, USA}
\email{tjeanlouis13@amherst.edu}

\author{Michael Young}
\address{Department of Mathematics, Iowa State University, Ames IA 50011, USA}
\email{myoung@iastate.edu}

\begin{abstract}
A \emph{mixed graph} is a graph with directed edges, called \emph{arcs}, and undirected edges. A $k$-coloring of the vertices is \emph{proper} if
colors from $\{1,2,\ldots,k\}$ are assigned to each vertex such that $u$ and $v$ have different colors if $uv$ is an edge, and the color of $u$ is
less than or equal to (resp.\ strictly less than) the color of $v$ if $uv$ is an arc. The \emph{weak} (resp.\ \emph{strong}) \emph{chromatic
polynomial} of a mixed graph counts the number of proper $k$-colorings. 
Using order polynomials of partially ordered sets, we establish a reciprocity theorem for weak chromatic polynomials giving
interpretations of evaluations at negative integers. 
\end{abstract}



\date{11 June 2013}

\thanks{We thank Thomas Zaslavsky and two anonymous referees for various helpful suggestions about this paper, and 
Ricardo Cortez and the staff at MSRI for creating an ideal research environment at MSRI-UP.
This research was partially supported by the NSF through the grants DMS-1162638 (Beck),
DMS-0946431 (Young), and DMS-1156499 (MSRI-UP REU), and by the NSA through grant H98230-11-1-0213.}

\maketitle


\section{Introduction}
A \emph{mixed graph} $G = (V, E, A)$ consists of a set of vertices, $V=V(G)$, a set of undirected edges, $E=E(G)$, and a set of directed edges, $A=A(G)$. For convenience, the
elements of $E$ will be called \emph{edges} and the elements of $A$ will be called \emph{arcs}. Given adjacent vertices $u,v \in V$, an edge will be denoted by $uv$ and an arc will be denoted by $\vec{uv}$.

A \emph{$k$-coloring} of a mixed graph $G$ is a mapping $c: V \rightarrow [k]$, where $[k] := \{1, 2, \ldots, k\}$. A \emph{weak} (resp.\
\emph{strong}) \emph{proper $k$-coloring} of $G$ is a $k$-coloring such that
\[
  c(u) \neq c(v) \ \text{ if } uv \in E
  \qquad \text{ and } \qquad 
  c(u) \le c(v) \ \text{ (resp. $c(u) < c(v)$) \ if } \vec{uv} \in A \, .
\]
The \emph{weak} (resp.\ \emph{strong}) \emph{chromatic polynomial}, denoted by $\chi_G(k)$ (resp.\ $\widehat{\chi}_G(k)$), is
the number of weak (resp. strong) proper $k$-colorings of $G$. It is well known (see, e.g.,
\cite{kotekmakowskyzilber,sotskovtanaev}) that these counting functions are indeed polynomials in $k$.
Coloring problems in mixed graphs have various applications, for example in scheduling problems in which one has both
disjunctive and precedence constraints (see, e.g., \cite{furmanczykkosowkiries,hansenkuplinskydewerra,Polyproof}).

An \emph{orientation} of a mixed graph $G$ is obtained by orienting the edges of $G$, i.e., assigning one of $u$ and $v$ to be the head/tail of the edge $uv \in E$; if $v$ is the head we use
the notation $u \to v$.
(An arc $\vec{uv}$, for which we also use the notation $u \rightarrow v$, cannot be re-oriented.)
An orientation of a mixed graph is \emph{acyclic} if it does not contain any directed cycles. A mixed graph is \emph{acyclic} if all of its
possible orientations are acyclic. A coloring $c$ and an orientation of $G$ are \emph{compatible} if for every $u \to v$ in the orientation, $c(u) \le c(v)$.

A famous theorem of Stanley says that, for any graph $G=(V,E,\emptyset)$ and positive integer $k$, $(-1)^{|V|} \chi_G(-k)$
enumerates the pairs of $k$-colorings and compatible acyclic orientations of $G$ and, in particular, $(-1)^{|V|}
\chi_G(-1)$ equals the number of acyclic orientations of $G$ \cite{stanleyacyclic}; this is an example of a
\emph{combinatorial reciprocity theorem}.
More recently, Beck, Bogart, and Pham proved the following analogue of Stanley's reciprocity
theorem for the strong chromatic polynomial of a mixed graph~\cite{Golomb}:

\begin{theorem}\label{strongmg} 
For any mixed graph $G=(V,E,A)$ and positive integer $k$, $(-1)^{|V|} \widehat{\chi}_G(-k)$ equals the number of $k$-colorings of $G$, each counted with multiplicity equal to the number of compatible acyclic orientations of~$G$.
\end{theorem}

In this paper, we complete the picture by proving a reciprocity theorem for \emph{weak} chromatic polynomials $\chi_G(k)$ of mixed graphs. A coloring $c$ and an orientation of $G$ are \emph{intercompatible} if for every $u \to v$ in the orientation,
\[
  c(u) \le c(v) \ \text{ if } uv \in E(G)
  \qquad \text{ and } \qquad
  c(u)<c(v) \ \text{ if } \vec{uv} \in A(G) \, . 
\]
Our main results is:

\begin{theorem}\label{main}
For any acyclic mixed graph $G=(V,E,A)$ and positive integer $k$, $(-1)^{|V|} {\chi}_G(-k)$ equals the number of $k$-colorings of $G$, each counted with multiplicity equal to the number of intercompatible acyclic orientations of~$G$.
\end{theorem}

One can prove this theorem along somewhat similar lines to the (geometric) approach used in \cite{Golomb}, though there are subtle details that distinguish the case of weak chromatic
polynomials from the one of strong chromatic polynomials. For example, although both Theorems \ref{strongmg} and \ref{main} result in relating $k$-colorings of a mixed graph to its acyclic
orientations, the reciprocity theorem for strong chromatic polynomials applies to all mixed graphs $G$, while the reciprocity theorem for weak chromatic polynomials requires the condition
that $G$ be an acyclic mixed graph: without this condition, Theorem \ref{main} is not true.

Our proof of Theorem \ref{main} applies Stanley's reciprocity theorem for \emph{order polynomials}, stated in Section \ref{pst},
which also contains the proof of Theorem \ref{main}. In Section \ref{delcon} we give a deletion--contraction method for computing
the weak and strong chromatic polynomials for mixed graphs, as well as an example that shows Theorem \ref{main} may not hold for mixed graphs that are not acyclic.


\section{Posets, Order Polynomials, and the Proof of Theorem \ref{main}}\label{pst}

Recall that a partially ordered set (a \emph{poset}) is a set $P$ with a relation $\preceq$ that is reflexive, antisymmetric, and
transitive.
Following \cite{stanleythesis} (see also \cite[Chapter 3]{Stanley}), we define an \emph{$\omega$-labeling} of a poset with $n$
elements as a bijection $\omega : P \rightarrow [n]$, and the \emph{order polynomial} $\Omega _{P,\omega}(k)$ as 
\[
  \Omega _{P,\omega} (k) := \# \left\{ (x_1, x_2, ... , x_n) \in [k]^n : 
  \begin{array}{l}
    x_u \leq x_v \text{ if } u \preceq v \text{ and } \omega (u) < \omega (v) \\
    x_u < x_v \text{ if } u \preceq v \text{ and } \omega (u) > \omega (v) 
  \end{array} \right\} .
\]
(Here $\# S$ denotes the cardinality of the set $S$.)
Stanley \cite{stanleythesis} proved that $\Omega _{P,\omega}(k)$ is indeed a polynomial in $k$. 
The \emph{complementary labeling} to $\omega$ 
is the $\overline{\omega}$-labeling of $P$ defined by $\overline{\omega}(v) := n+1-\omega(v)$.
Thus
\[
  \Omega _{P,\overline{\omega}} (k) = \# \left\{ (x_1, x_2, ... , x_n) \in [k]^n : 
  \begin{array}{l}
    x_u < x_v \text{ if } u \preceq v \text{ and } \omega (u) < \omega (v) \\
    x_u \le x_v \text{ if } u \preceq v \text{ and } \omega (u) > \omega (v) 
  \end{array} \right\} .
\]
  
\begin{theorem}[Stanley \cite{stanleythesis}]\label{stanley_op}
$
  \Omega _{P,\omega} (-k) = (-1)^{|P|} \, \Omega_{P,\overline{\omega}} (k) \, .
$
\end{theorem}

\noindent
The reciprocity relation given in Theorem \ref{stanley_op} takes on a special form when $\omega$ is a \emph{natural labeling} of $P$, i.e., one that respects the order of $P$.
(It is easy to see that every poset has a natural labeling.)
In this case $\Omega _{P,\omega} (k)$ simply counts all order preserving maps $x: P \to [k]$ (i.e., $u \preceq v \Longrightarrow x_u \le x_v$), whereas $\Omega_{P,\overline{\omega}} (k)$ couns all \emph{strictly} order preserving maps $x: P \to [k]$ (i.e., $u \prec v \Longrightarrow x_u < x_v$). 
Theorem \ref{stanley_op} implies that these two counting functions are reciprocal.

For a mixed graph $G=(V,E,A)$ with $n$ vertices, the weak chromatic polynomial $\chi _G(k)$ can be written as 
\[
  \chi _G(k) = \# \left\{ (x_1, x_2, \dots, x_n) \in [k]^n :
  \begin{array}{l}
    x_u \leq x_v \text{ if } \vec{uv} \in A \\
    x_u \ne x_v \text{ if } uv \in E
  \end{array} \right\} .
\]
Each acyclic orientation of $G$ can be translated into a poset by letting $P = V(G)$ and introducing, for each $u \to v$ in the orientation, the relation $u \preceq v$. 

Throughout the remainder of this section, we fix an acyclic mixed graph $G$ and denote by $G_1, G_2, \ldots, G_m$ the (acyclic) orientations of $G$. For each $1 \le i \le m$, denote $P_i$ as the poset created by the orientation $G_i$, and
let $\phi_{G_i}(k)$ be the number of weak proper $k$-colorings of $G_i$ that are also weak proper $k$-colorings of $G$.

\begin{lemma}\label{phisum}
If $G$ is an acyclic mixed graph, then $\chi_G(k) = \displaystyle\sum_{i=1}^m \phi_{G_i}(k) \, .$
\end{lemma}

\begin{proof}
It is clear that each weak proper $k$-coloring of $G$ is a weak proper $k$-coloring of $G_i$ for some $1 \le i \le m$.
Conversely, assuming $E(G) \ne \emptyset$, for any $1 \le i < j \le k$, there is some $uv \in E(G)$ such that $u
\to v$ in $G_i$ and $v \to u$ in $G_j$. This implies that there is no weak proper coloring that is a weak proper $k$-coloring of $G_i$ and~$G_j$.
If $E(G) = \emptyset$ then $G$ is the only orientation of itself.
\end{proof}

\begin{lemma}\label{labeling}
For each $G_i$, there exists an $\omega_i$-labeling of $P_i$ such that $$\phi_{G_i}(k) = \Omega_{P_i,\omega_i}(k) \, .$$ Moreover, $\Omega_{P_i,\overline{\omega}_i}(k)$ is the number of $k$-colorings intercompatible with $G_i$.
\end{lemma}

\begin{proof}
Given the orientation $G_i$, let $R_i$ be the orientation of $G$ obtained by reversing the orientation of the edges in $G_i$ (but not the arcs). 
We will construct $\omega_i$ recursively.

Since $R_i$ is acyclic, there exists a vertex $v \in V$ such that all edges and arcs incident to $v$ are oriented away from it. 
Set $\omega_i(v) := 1$ and remove $v$ and the arcs incident to $v$.  Since $R_i$ is acyclic, $R_i-v$ must also be acyclic.  
Now repeat, assigning each vertex in the process consecutive $\omega_i$-labels.
This gives $\omega_i$-labels that satisfy
\[
  u \to v \text{ in } R_i
  \qquad \Longrightarrow \qquad
  \omega_i(u) < \omega_i(v) \, ,
\] 
resulting in an $\omega_i$-labeling of $P_i$, the poset corresponding to $G_i$, that satisfies for $u \preceq v$ 
\begin{align*}
  \omega_i(u) < \omega_i(v) \qquad &\Longrightarrow \qquad \vec{uv} \in A(G) \, , \\
  \omega_i(u) > \omega_i(v) \qquad &\Longrightarrow \qquad uv \in E(G) \, .
\end{align*}
So
\begin{align*}
  \Omega_{P_i,\omega_i}(k)
  &= \# \left\{ (x_1, x_2, ... , x_n) \in [k]^n : 
  \begin{array}{l}
    x_u \le x_v \text{ if } u \preceq v \text{ and } \omega_i (u) < \omega_i (v) \\
    x_u < x_v \text{ if } u \preceq v \text{ and } \omega_i (u) > \omega_i (v) 
  \end{array} \right\} \\
  &= \# \left\{ (x_1, x_2, ... , x_n) \in [k]^n : 
  \begin{array}{l}
    x_u \le x_v \text{ if } u \to v \text{ in } G_i \text{ and } \vec{uv} \in A(G) \\
    x_u < x_v \text{ if } u \to v \text{ in } G_i \text{ and } uv \in E(G)
  \end{array} \right\} \\ 
  &= \phi_{G_i}(k) \, .
\end{align*}
For the second part of the proof, recall that
\begin{align*}
  \Omega _{P_i,\overline{\omega_i}} (k)
  &= \# \left\{ (x_1, x_2, ... , x_n) \in [k]^n : 
  \begin{array}{l}
    x_u < x_v \text{ if } u \preceq v \text{ and } \omega_i (u) < \omega_i (v) \\
    x_u \le x_v \text{ if } u \preceq v \text{ and } \omega_i (u) > \omega_i (v) 
  \end{array} \right\} \\
  &= \# \left\{ (x_1, x_2, ... , x_n) \in [k]^n : 
  \begin{array}{l}
    x_u < x_v \text{ if } u \preceq v \text{ and } \vec{uv} \in A(G) \\
    x_u \le x_v \text{ if } u \preceq v \text{ and } uv \in E(G)
  \end{array} \right\} \\
  &= \# \text{ colorings intercompatible with } G_i \, . \qedhere
\end{align*}
\end{proof}

\begin{proof}[Proof of Theorem \ref{main}]
If $G$ is an acyclic mixed graph, then by Lemma \ref{phisum},
\begin{eqnarray*}
\chi_G(-k) &=& \displaystyle\sum_{i=1}^m \phi_{G_i}(-k) \\
&=& \displaystyle\sum_{i=1}^m \Omega_{P_i,\omega_i}(-k) \ \text{ (by Lemma \ref{labeling})}\\
&=& \displaystyle\sum_{i=1}^m (-1)^{|P_i|}\Omega_{P_i,\overline{\omega_i}}(k) \ \text{ (by Theorem \ref{stanley_op})}\\
&=& (-1)^{|V|}\displaystyle\sum_{i=1}^m \Omega_{P_i,\overline{\omega_i}}(k) \, . 
\end{eqnarray*}
By applying Lemma \ref{labeling} again, the proof is completed.
\end{proof}


\section{Deletion--Contraction Computations}\label{delcon}

Let $G=(V,E,A)$ be a mixed graph, $e \in E(G)$, and $a \in A(G)$. Define $G-e = (V,E-e,A)$ as the mixed graph with edge $e$ deleted and $G-a = (V,E,A-a)$ as the mixed graph with arc $a$ deleted. An edge or arc is \emph{contracted} by deleting the edge or arc and identifying the vertices incident to it (keeping only one copy of each edge and arc). Denote $G/e$ as the mixed graph obtained by contracting edge $e$ in $G$ and $G/a$ as the mixed graph obtained by contracting arc $a$ in $G$.
The standard proof for the deletion--contraction formula for (unmixed) graphs gives:

\begin{proposition}\label{DCedge} 
If $G$ is a mixed graph and $e \in E(G)$, then $$\chi_G(k) = \chi_{G-e}(k) - \chi_{G/e}(k) \, .$$
\end{proposition}


Define $G_{a}$ as the mixed graph $G$ with arc $a$ directed in the reverse direction. In other words, if $a=\vec{uv}$ then $G_{a} = (V,E,A-\{\vec{uv}\} \cup \{\vec{vu}\})$.  

\begin{proposition}\label{DCarc} 
If $G$ is a mixed graph and $a \in A(G)$, then $$\chi_G(k) + \chi_{G_{a}}(k) = \chi_{G-a}(k) + \chi_{G/a}(k) \, .$$
\end{proposition}

\begin{proof}
Let $a=\vec{uv}$, $C$ be the set of weak proper $k$-colorings of $G$, and $C_a$ be the set of weak proper $k$-colorings of $G_a$. Therefore, $\chi_G(k) + \chi_{G_{a}}(k) = |C \cup C_a| + |C \cap C_a|$.

A coloring $c \in C \cup C_a$ if and only if $c$ is a weak proper $k$-coloring of $G-a$. A coloring $c \in C \cap C_a$ if and only if $c(u) = c(v)$ and $c$ corresponds to a weak proper $k$-coloring of $G/a$ in which the vertex created by identifying $u$ and $v$ is colored with $c(u)$. Therefore, $\chi_{G-a}(k) = |C \cup C_a|$ and $\chi_{G/a}(k) = |C \cap C_a|.$
%
\end{proof}

Propositions \ref{DCedge} and \ref{DCarc} give the following equations:
\begin{align}
\label{dce} \chi_G(k) &= \chi_{G-e}(k) - \chi_{G/e}(k) \, , \\
\label{dca} \chi_G(k) &= \chi_{G-a}(k) + \chi_{G/a}(k) - \chi_{G_{a}}(k) \, .
\end{align} 
Equation (\ref{dce}) is very useful in computing the weak chromatic polynomials since it recursively gives the weak chromatic polynomial of a mixed graph as a
difference of (in the number of vertices or edges) smaller mixed graphs. On the other hand, equation (\ref{dca}) gives the weak chromatic polynomial of $G$ in terms of $G_a$, which is not a smaller graph. However, we will show how it can be used in computation.

A directed graph $G=(V,\emptyset, A)$ is \emph{strongly connected} if for any pair of vertices $u,v \in V$ there exists a directed path from $u$ to $v$.

\begin{proposition}\label{cycle_contract_lemma}
If $G$ is a strongly connected directed graph, then $\chi_G(k) = k$.
\end{proposition}

\begin{proof}
Fix $u \in V$, and let $c$ be a weak proper coloring of $G$. Since there is a directed path from $u$ to any $v$ and vice versa, $c(v) \le c(u) \le c(v)$ for every
$v \in V$. Therefore, $c(u) = c(v)$ for every $v \in V$, and since there are $k$ colors that can be assigned to $u$, $\chi_G(k) = k$.
\end{proof}

Given a subgraph $S$ of $G$, denote $G/S$ as the mixed graph $G$ with all edges and arcs of $S$ removed and all
vertices of $S$ identified to one vertex; resulting parallel edges/arcs should be replaced by a single edge/arc.

\begin{proposition}\label{subcontract}
Let $G$ be a mixed graph and $S$ be a strongly connected directed subgraph of $G$. Then $\chi_G(k)=\chi_{G/S}(k)$. 
\end{proposition}

\begin{proof}
Let $s$ be the vertex that $S$ contracts to in $G/S$. For each weak proper $k$-coloring of $G$, the vertices of $S$ must all be colored the same color $j$. By defining $c(s) = j$ we get a bijection between the weak proper $k$-colorings of $G$ and $G/S$.
\end{proof}

Computing the weak chromatic polynomial of a mixed graph is reduced to computing the weak chromatic polynomial of smaller directed graphs by applying Proposition
\ref{DCedge}. Computing the weak chromatic polynomial of a directed graph is reduced to computing the weak chromatic polynomial of smaller acyclic directed graphs
(directed trees) by recursively reversing arcs and applying Proposition \ref{DCarc} until a strongly connected subgraph is created and Proposition \ref{subcontract} can be applied. Note that a strongly connected subgraph of a directed graph can be obtained by reversing arcs as long as the underlying graph is not acyclic.

\begin{figure}[h]
\begin{center}
\includegraphics[height=1.75in]{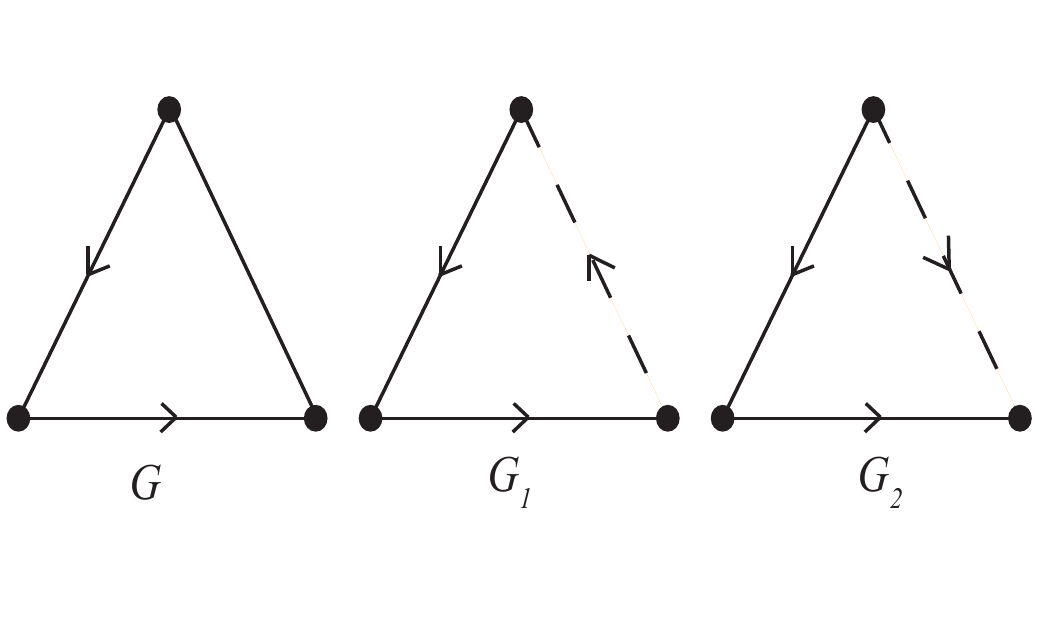}
\end{center}
\caption{A mixed graph $G$ and its two orientations.}
\label{ex1}
\end{figure}

As an example, let $G=(\{u,v,w\}, \{uv\}, \{\vec{vw},\vec{wu} \})$ (shown in Figure \ref{ex1}). $G$ is a cyclic mixed graph since it has an orientation, $G_1$, that contains a directed cycle. Consider $k=2$. If $c$ is an intercompatible coloring of $G_1$ or $G_2$, then $c(v) < c(w) < c(u)$. Therefore, $G_1$ and $G_2$ have no intercompatible colorings and the number of 2-colorings of $G$, each counted with multiplicity equal to the number of intercompatible acyclic orientations of $G$ is~0.

\begin{figure}[h]
\begin{center}
\includegraphics[height=1.5in]{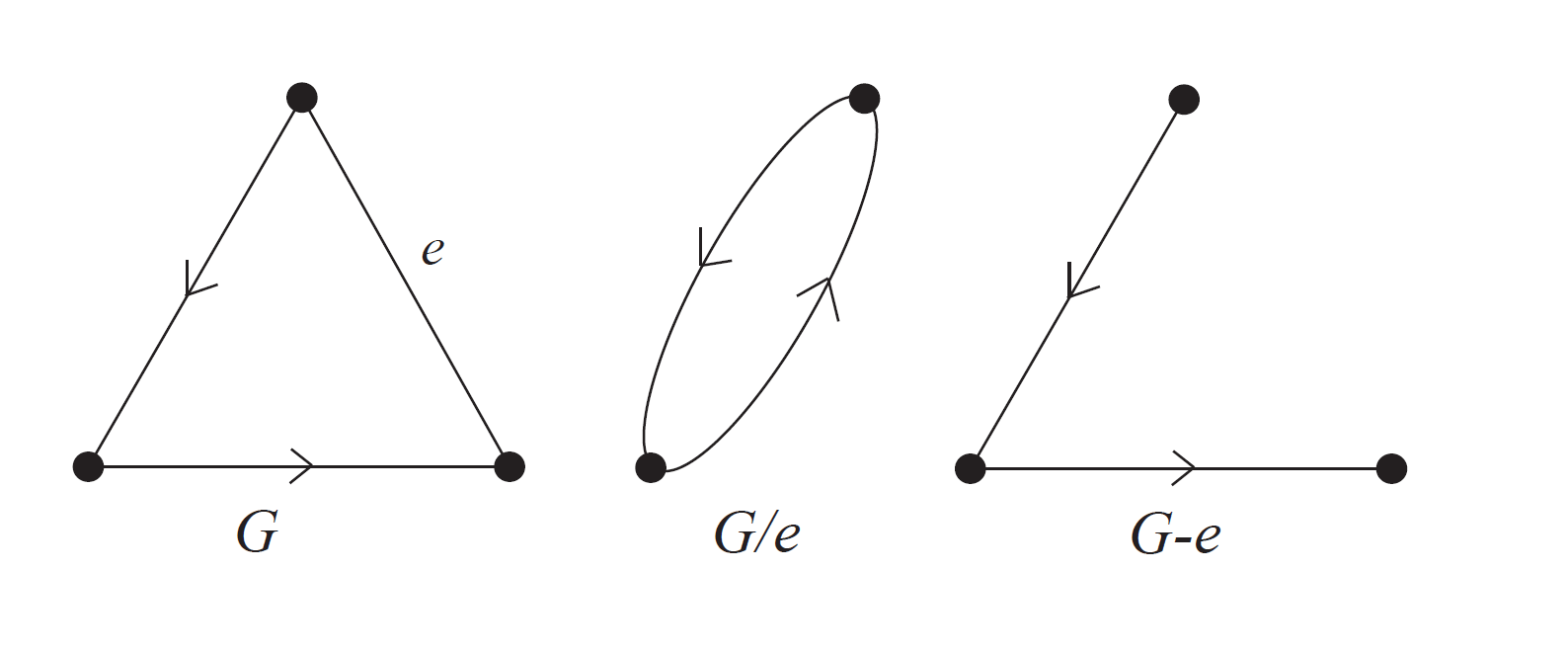}
\end{center}
\caption{$G$ and its contraction and deletion.}
\label{ex2}
\end{figure}

We now use contraction and deletion, with $e=uv$, to compute the weak chromatic polynomial of $G$. The contracted
graph $G/e$ (see Figure \ref{ex2}) is a strongly connected directed graph, so $\chi_{G/e}(k) = k$. In $G-e$, there are $(k-i+1)i$ weak proper $k$-colorings with $c(w)=i$. Therefore,
\[
 \chi_{G-e}(k) = \sum_{i=1}^k (k-i+1)i = \frac{(k+2)(k+1)k}{3}
\]
and so $\chi_G(k) = \frac 1 3 (k+2)(k+1)k - k$. We can now see that Theorem \ref{main} does not hold for $G$ since $\chi_G(-2) = -2$.


\bibliographystyle{plain}
\bibliography{weakMG} 

\begin{thebibliography}{1}

\bibitem{Golomb}
Matthias Beck, Tristram Bogart, and Tu~Pham.
\newblock Enumeration of {G}olomb rulers and acyclic orientations of mixed
  graphs.
\newblock {\em Electronic J. Combinatorics}, 19:P42, 2012.

\bibitem{furmanczykkosowkiries}
Hanna Furma{\'n}czyk, Adrian Kosowski, Bernard Ries, and Pawe{\l}
  {\.Z}yli{\'n}ski.
\newblock Mixed graph edge coloring.
\newblock {\em Discrete Math.}, 309(12):4027--4036, 2009.

\bibitem{hansenkuplinskydewerra}
Pierre Hansen, Julio Kuplinsky, and Dominique de~Werra.
\newblock Mixed graph colorings.
\newblock {\em Math. Methods Oper. Res.}, 45(1):145--160, 1997.

\bibitem{kotekmakowskyzilber}
Tomer Kotek, Johann~A. Makowsky, and Boris Zilber.
\newblock On counting generalized colorings.
\newblock In {\em Computer science logic}, volume 5213 of {\em Lecture Notes in
  Comput. Sci.}, pages 339--353. Springer, Berlin, 2008.

\bibitem{sotskovtanaev}
Yuri~N. Sotskov and Vjacheslav~S. Tanaev.
\newblock Chromatic polynomial of a mixed graph.
\newblock {\em Vesc\=\i\ Akad. Navuk BSSR Ser. F\=\i z.-Mat. Navuk},
  (6):20--23, 140, 1976.

\bibitem{Polyproof}
Yuri~N. Sotskov, Vjacheslav~S. Tanaev, and Frank Werner.
\newblock Scheduling problems and mixed graph colorings.
\newblock {\em Optimization}, 51(3):597--624, 2002.

\bibitem{stanleythesis}
Richard~P. Stanley.
\newblock {\em Ordered structures and partitions}.
\newblock American Mathematical Society, Providence, R.I., 1972.
\newblock Memoirs of the American Mathematical Society, No. 119.

\bibitem{stanleyacyclic}
Richard~P. Stanley.
\newblock Acyclic orientations of graphs.
\newblock {\em Discrete Math.}, 5:171--178, 1973.

\bibitem{Stanley}
Richard~P. Stanley.
\newblock {\em Enumerative {C}ombinatorics. {V}olume 1}, volume~49 of {\em
  Cambridge Studies in Advanced Mathematics}.
\newblock Cambridge University Press, Cambridge, second edition, 2012.

\end{thebibliography}

\end{document}